\documentclass{article}%[,12pt,oneside,a4paper,brazil,english,sumario=tradicional]
\usepackage[colorlinks=true,linkcolor=black,citecolor=black,filecolor=black,linkcolor=black,urlcolor=black]{hyperref}
\usepackage{relsize}
\usepackage[T1]{fontenc}
\usepackage{lmodern}
\usepackage[utf8]{inputenc}%Codificacao do documento
\usepackage{indentfirst}%Indenta o primeiro parágrafo de cada seção.
\usepackage{nomencl}%Lista de simbolos
\usepackage{color}%Controle das cores
\usepackage{graphicx}%Inclusão de gráficos
\usepackage{microtype}%Para melhorias de justificação
\usepackage{lipsum}%Para geração de dummy text
\usepackage{multicol}										
\usepackage{mathtools}
\usepackage{latexsym}				
\usepackage{amssymb,amsmath,amsthm,amsfonts,amscd,amstext,amsbsy}
\usepackage{enumerate}
\usepackage{enumitem}
\usepackage{todonotes}

\newcommand{\Q}{\mathfrak{Q}}
\newcommand{\Hc}{\mathcal{H}}
\newcommand{\M}{\mathcal{M}}
\DeclareMathOperator{\diver}{div}
\DeclareMathOperator{\cut}{cut}
\DeclareMathOperator{\Hess}{\nabla^2}
\newcommand{\Hessij}{\nabla^2_{ij}}
\DeclareMathOperator{\dist}{dist}

\DeclareMathOperator{\diam}{diam}
\DeclareMathOperator{\R}{\mathbb{R}}
\DeclareMathOperator{\HH}{\mathbb{H}}
\DeclareMathOperator{\Ricc}{Ricc}
\newcommand{\W}{\Omega}
\newcommand{\norm}[1]{\left\| #1 \right\|}
\newcommand{\modulo}[1]{\left| #1 \right|}
\usepackage{mathrsfs}
\newcommand{\cl}{\mathscr{C}}
\newcommand{\ds}{\displaystyle}
\newcommand{\Ei}{\ds\mathsmaller{\frac{\partial}{\partial x_i}}}
\newcommand{\Ej}{\ds\mathsmaller{\frac{\partial}{\partial x_j}}}
\newcommand{\Dz}{\partial_z}

\newcommand{\Dj}{\partial_j}

\newcommand{\parcial}[2]{\frac{\partial#1}{\partial#2}}

\newcommand{\escalar}[2]{{\left\langle #1,#2 \right\rangle}}

\newcommand{\lrarrow}{\longrightarrow}
\newcommand{\funciones}[5]{\begin{array}{rccl}
#1:&#2&\lrarrow&#3\\
&#4& \xmapsto{\phantom{\lrarrow}}&#5\end{array}}

\newtheorem{teo}{Theorem}%[section]
\newtheorem{obs}[teo]{Remark}
\newtheorem{prop}[teo]{Proposition}
\newtheorem{lema}[teo]{Lemma}
\newtheorem{cor}[teo]{Corollary}

\theoremstyle{definition}
\newtheorem{defi}[teo]{Definition}

\usepackage{amsthm}
\makeatletter
\def\th@plain{%
  \thm@notefont{}% same as heading font
  \itshape % body font
}
\def\th@definition{%
  \thm@notefont{}% same as heading font
  \normalfont % body font
}
\makeatother

\makeatletter
\let\oldr@@t\r@@t
\def\r@@t#1#2{%
\setbox0=\hbox{$\:\oldr@@t#1{#2\,}$}\dimen0=\ht0
\advance\dimen0-0.2\ht0
\setbox2=\hbox{\vrule height\ht0 depth -\dimen0}%
{\box0\lower0.4pt\box2}}
\LetLtxMacro{\oldsqrt}{\sqrt}
\renewcommand*{\sqrt}[2][]{\oldsqrt[#1]{#2}}
\makeatother

\newcommand\blfootnote[1]{%
  \begingroup
  \renewcommand\thefootnote{}\footnote{#1}%
  \addtocounter{footnote}{-1}%
  \endgroup
}

\begin{document}
\title{Sharp solvability criteria for Dirichlet problems of mean curvature type in Riemannian manifolds: non-existence results}
\author{Yunelsy N Alvarez\thanks{Supported by {CAPES and} CNPq of Brazil.} \and	Ricardo Sa Earp\thanks{Partially supported by CNPq of Brazil.}}
\maketitle

\blfootnote{\emph{Acknowledgments.} The authors would like to thank the referee for the careful reading and the valuable and useful
suggestions.}

\begin{abstract}
It is well known that the \textit{Serrin condition} is a necessary condition for the solvability of the Dirichlet problem for the prescribed mean curvature equation in bounded domains of $\R^n$ with certain regularity. In this paper we investigate the sharpness of the Serrin condition for the vertical mean curvature equation in the product $ M^n \times \R $. Precisely, given a $\cl^2$ bounded domain $\W$ in $M$ and a function $ H = H (x, z) $ continuous in $\overline{\W}\times\R$ and non-decreasing in the variable $z$, we prove that the \textit{strong Serrin condition} $(n-1)\Hc_{\partial\W}(y)\geq n\sup\limits_{z\in\R}\modulo{H(y,z)} \ \forall \ y\in\partial\W $, is a necessary condition for the solvability of the Dirichlet problem in a large class of Riemannian manifolds within which are the Hadamard manifolds and manifolds whose sectional curvatures are bounded above by a positive constant. As a consequence of our results we deduce Jenkins-Serrin and Serrin type sharp solvability criteria.
\end{abstract}

\section{Introduction}
\markboth{Y. N. Alvarez and R. Sa Earp}{Introduction}
\label{secIntroducao}

\medskip

We denote by $M$ a complete Riemannian manifold of dimension $n\geq 2$ and let $\W$ be a domain in $M$. The focus of our work is the prescribed mean curvature equation for vertical graphs in $M\times\R$, that is,
\begin{equation}\label{operador_minimo_1}%{PD_curv_media_div}%{operador_minimo_1}
%\M u:=\diver_M\left(\dfrac{\nabla_M u}{\sqrt{1+\norm{\nabla_M u}_M^2}}\right) = nH(x,u),
\M u:=\diver \left(\dfrac{\nabla u}{W}\right) = nH(x,u),
\end{equation}
where $H$ is a continuous function over $\overline{\W}\times\R$ and it is non-decreasing in the variable $z$, $W=\sqrt{1+\norm{\nabla u (x)}^2}$ and {the gradient, the norm and divergence} are calculated with respect to the metric of $M$. In a coordinates system $(x_1,\dots,x_n)$ in $M$, it follows that
\begin{equation}\label{operador_minimo_1_coord}
\M u=\ds\frac{1}{W}\sum_{i,j=1}^n \left(\sigma^{ij} - \frac{u^iu^j}{W^2} \right)\Hessij u=nH(x,u),
\end{equation}
where $(\sigma^{ij})$ is the inverse of the metric $(\sigma_{ij})$ of $M$, $u^i=\ds\sum_{j=1}^n\sigma^{ij} \Dj u$ are the coordinates of $\nabla u$ and $\Hessij u(x)=\Hess u(x){\left(\Ei,\Ej\right)}$. We will denote by $\Q$ the operator defined by
$$%\begin{equation}\label{oper_Q}
\Q u = \M u -n H(x,u).
$$%\end{equation}
%\todo[inline]{Devo colocar aqui alguma coisa ao respeito de um operador quasilinear e os coeficientes.}
We notice that the coefficient matrix of the operator $\M$ {(that is, the matrix whose entries are the coefficients of the second derivatives)} is given by $A=\frac{1}{W}g$, where $g$ is the induced metric on the graph of $u$. This implies that the eigenvalues of $A$ are positive and depend on $x$ and on $\nabla u$. Hence, $\M$ is locally uniformly elliptic. Furthermore, if $\W$ is bounded and $u\in\cl^{1}(\overline{\W})$, then $\M$ is uniformly elliptic in $\overline{\W}$ (see \cite{spruck} for more details).

It has been proved in chronological order by Finn \cite{Finn1965}, Jenkins-Serrin {\cite{Serrin1968}} and Serrin \cite{Serrin}, that the very well known Serrin condition is a necessary condition for the solvability of the Dirichlet problem for equation
\eqref{operador_minimo_1} in bounded domains of $\R^n$. 

Dirichlet problems for equations whose solutions describe hypersurfaces of prescribed mean curvature have been studied in manifolds different from the Euclidean space. Several works have considered Serrin type conditions that provide some existence theorems (see \cite{Aiolfi}, \cite{Alias2008}, \cite{Dajczer2008}, \cite{Dajczer2005}, \cite{eliasarticle},  \cite{Lopez2001}, \cite{Nitsche2002} and \cite{spruck} as examples). However, non-existence theorems have been only investigated in a few cases that we summarize below. %To the best of our knowledge, there are only three non-existence result showing that the Serrin type condition is sharp.

For instance, P.-A Nitsche \cite{Nitsche2002} was concerned with graph-like prescribed mean curvature hypersurfaces in hyperbolic space $\HH^{n+1}$. In the half-space setting, he studied radial graphs over the totally geodesic hypersurface $S = \{x \in \R^{n+1}_+; (x_0)^2 + \dots + (x_n)^2 = 1\}$. He established an existence result if $\W$ is a bounded domain of $S$ of class $\cl^{2,\alpha}$ and $H\in\cl^1(\overline{\W})$ is a function satisfying $\sup\limits_{\overline{\W}}\modulo{H}\leq 1$ and $\modulo{H(y)}<\Hc_{C}(y)$ everywhere on $\partial\W$, where $\Hc_{C}$ denotes the hyperbolic mean curvature of the cylinder $C$ {spanned by the rays issuing from the origin of $\R^{n+1}$ and intersecting $\partial\W$}. Furthermore, he showed the existence of smooth boundary data such that no solution exists in case of $\modulo{H(y)}>\Hc_{C}(y)$ for some $y\in\partial\W$ under the assumption that $H$ has a sign. We observe that {these} results do not provide a Serrin type solvability criterion.

Also in the half-space model of $\HH^{n+1}$, E. M. Guio-R. Sa Earp \cite{elias,eliasarticle} considered a bounded domain $\W$ contained in a vertical totally geodesic hyperplane $P$ of $\HH^{n+1}$ and studied the Dirichlet problem for the mean curvature equation for horizontal graphs over $\W$, that is, hypersurfaces which intersect at most only once the horizontal horocycles orthogonal to $\W$. They considered the hyperbolic cylinder $C$ generated by horocycles cutting ortogonally $P$ along the boundary of $\W$ and the Serrin condition, $\Hc_C(y) \geq \modulo{H(y)}$ $\forall\ y\in\partial\W$. They obtained a Serrin type solvability criterion for prescribed mean curvature $H=H(x)$ and also proved a sharp solvability criterion for constant $H$.

Finally, M. Telichevesky {\cite[Th. 6 p. 246]{miriam}} proved that if $M$ is a Hadamard manifold {whose sectional curvature is bounded above by $-1$}, then mean convexity is a necessary condition for the existence of a vertical minimal graphs in $M\times\R$ over a domain $\W$ of $M$ possibly unbounded. {The combination of this result} with an existence result of Aiolfi-Ripoll-Soret {\cite[Th. 1 p. 72]{Aiolfi}} gives sharp solvability criterion for the minimal hypersurface equation in bounded domains of $M$.

To the best of our knowledge, no other {non-existence result and} Serrin-type solvability criterion have been proved {in settings different from the Euclidean one. 

As a direct consequence of the main result of this paper, Theorem \ref{M_nao_exist_MxR_Hxz}, the aforementioned result in the $M\times\R$ context is generalized.  
%to every Cartan-Hadamard manifold. 
More precisely, the combination of the existence result of Aiolfi-Ripoll-Soret {\cite[Th. 1 p. 72]{Aiolfi}} for the minimal case with Collorary \ref{M_nao_exist_MxR_Hxz_Hadamard} shows that the sharp solvability criterion of Jenkins-Serrin {\cite[Th. 1 p. 171]{Serrin1968}} actually holds in every Cartan-Hadamard manifold:
\begin{teo}[Sharp Jenkins-Serrin-type solvability criterion]\label{cond_nece_suf_minimo_Hadamard}
Let $M$ be a Cartan-Hadamard manifold and $\W\subset M$ a bounded domain whose boundary is of class $\cl^{2,\alpha}$ for some $\alpha\in(0,1)$. Then the Dirichlet problem for equation $\M u=0$ in $\W$ has a unique solution for arbitrary continuous boundary data if, and only if, $\W$ is {mean convex}.
\end{teo}
%%%%%%%%%%%%%%%%%%%%%%%%%%%%%%%
Furthermore, a sharp {Serrin type result {\cite[p. 416]{Serrin}}} for constant mean curvature vertical graphs is inferred by combining our Corollary \ref{M_nao_exist_MxR_Hxz_positive_curvature} with an existence result of Spruck {\cite[Th. 1.4 p. 787]{spruck}}:
%%%%%%%%%%%%%%%%%%%%%%%%%%%%%%
\begin{teo}[Sharp Serrin-type solvability criterion]\label{cond_nece_suf_minimo_curv_positivo_este_este}
Let $M$ be a simply connected and compact manifold whose sectional curvature satisfies $\frac{1}{4} K_0 < K \leq K_0$ for a positive constant $K_0$. Let $\W\subset M$ be a domain with $\diam(\W)<\frac{\pi}{2\sqrt{K_0}}$ and whose boundary is of class $\cl^{2,\alpha}$ for some $\alpha\in(0,1)$. Then for every constant $H$ the Dirichlet problem for equation 
%\eqref{operador_minimo_1} 
$\M u=nH$ in $\W$ has a unique solution for arbitrary continuous boundary data if, and only if, $(n-1)\Hc_{\partial\W}\geq n\modulo{H}$.
\end{teo}

Before stating the main result we need to introduce the concept of radial curvatures. 
\begin{defi}[{Greene-Wu \cite[p. 5]{greenewu}}]
Let $M$ be a complete Riemannian manifold and let $y_0\in M$ be a fixed point. A {\it radial plane} $\Pi_x$ at a point $x\in M $ is a two dimensional subspace of $T_{x} M$ containing a vector tangent to a minimizing geodesic segment $\beta$ emanating from $y_0$. The {\it radial sectional curvature} with respect to the radial plane $\Pi_x$ is the sectional curvature $K(\Pi_x)$. We say that the radial curvature of $M$ along the geodesic segment $\beta$ is bounded above by a constant $K_0$ if $K(\Pi_x)\leq  K_0$ for every radial plane $\Pi_x\subset T_x M$ and every point $x\in[\beta]$. 
\end{defi}
%We denote by $\cut(y_0)$ the \textit{cut locus} of $y_0$.}:
}

\begin{teo}[main theorem]\label{M_nao_exist_MxR_Hxz}
Let $\W\subset M$ be a bounded domain whose boundary is of class $\cl^2$. Let $H\in\cl^0(\overline{\W}\times\R)$ be a function either non-positive or non-negative and non-decreasing in the variable $z$. Let us assume that there exists $y_0\in\partial\W$ such that
$$(n-1) \Hc_{\partial\W}(y_0) < n \ds\sup_{z\in\R}\modulo{H(y_0,z)}.$$
Suppose also that $\cut(y_0)\cap\W=\emptyset$. Furthermore, assume that the radial curvature over the radial geodesics segments issuing from $y_0$ and intersecting $\W$ is bounded above by $K_0$, where
\begin{enumerate}[label=(\alph*)]
\item $K_0\leq 0$, or
\item $K_0>0$ and $\dist(y_0,x)<\dfrac{\pi}{2\sqrt{K_0}}$ for all $x\in\overline{\W}$.
\end{enumerate}
Then there exists $\varphi\in\cl^{\infty}(\overline{\W})$ such that there is no $u \in \cl^0(\overline{\Omega})\cap \cl^2 (\Omega)$ satisfying equation \eqref{operador_minimo_1} with $u = \varphi$ in $\partial \Omega$.
\end{teo}
The statement ensures that the \textit{strong Serrin condition}
\begin{equation}\label{SerrinConditionGeneral}%\tag{CSF}
(n-1)\Hc_{\partial\W}(y)\geq n\ds\sup_{z\in\R}\modulo{H(y,z)} \ \forall \ y\in\partial\W
\end{equation}
is a necessary condition for the solvability of the Dirichlet problem for equation \eqref{operador_minimo_1}.

{As a direct consequence of item (a) in Theorem \ref{M_nao_exist_MxR_Hxz} we infer the following result in Hadamard manifolds.} %Some direct consequences inferred from our main non-existence theorem are stated as follows.
\begin{cor}\label{M_nao_exist_MxR_Hxz_Hadamard}
Let $M$ be a Cartan-Hadamard manifold and $\W\subset M$ a bounded domain whose boundary is of class $\cl^2$. Let
$H\in\cl^0(\overline{\W}\times\R)$ be a function either non-negative or non-positive and non-decreasing in the variable $z$. Suppose there exists $y_0\in\partial\W$ such that
$$(n-1) \Hc_{\partial\W}(y_0) < n \ds\sup_{z\in\R}\modulo{H(y_0,z)}.$$
Then there exists $\varphi\in\cl^{\infty}(\overline{\W})$ such that there is no $u \in \cl^0(\overline{\Omega})\cap \cl^2 (\Omega)$ satisfying equation \eqref{operador_minimo_1} with $u = \varphi$ in $\partial \Omega$.
\end{cor}

{Furthermore, from statement (b) %in Theorem \ref{M_nao_exist_MxR_Hxz} 
we derive the following non-existence result for a class of positively curved manifolds.}
\begin{cor}\label{M_nao_exist_MxR_Hxz_positive_curvature}
Let $M$ be a simply connected and compact manifold whose sectional curvature satisfies $\frac{1}{4} K_0 < K \leq K_0$ for a positive constant $K_0$. Let $\W\subset M$ be a domain with $\diam(\W)<\frac{\pi}{2\sqrt{K_0}}$ and whose boundary is of class $\cl^2$. Let $H\in\cl^0(\overline{\W}\times\R)$ be a function either non-negative or non-positive and non-decreasing in the variable $z$. Suppose there exists $y_0\in\partial\W$ such that
$$(n-1)\Hc_{\partial\W}(y_0) < n \ds\sup_{z\in\R}\modulo{H(y_0,z)}.$$
Then there exists $\varphi\in\cl^{\infty}(\overline{\W})$ such that there is no $u \in \cl^0(\overline{\Omega})\cap \cl^2 (\Omega)$ satisfying equation \eqref{operador_minimo_1} with $u = \varphi$ in $\partial \Omega$.
\end{cor}
{We remark that the assumptions on $M$ in the above statement guarantee that the injectivity radius of $M$ is greater than or equal to $\frac{\pi}{\sqrt{K_0}}$, thus $\cut(y_0)\cap\W=\emptyset$ since $\diam(\W)<\frac{\pi}{2\sqrt{K_0}}$.}

\section{Further sharp solvability criteria}
%\markboth{Y. N. Alvarez and R. Sa Earp}{Consequences of our non-existence results}
%We now want to highlight Serrin type solvability criteria derived from the combination of our non-existence results with existence results obtained by others \cite{spruck,Aiolfi} and by the authors \cite{minhatese,artigoexist:inpress}.

%We now want to highlight {further} Serrin type solvability criteria derived from the combination of our non-existence results with existence results obtained by {Spruck \cite{spruck}} and by the authors \cite{artigoexist:inpress}.%\cite{minhatese,artigoexist:inpress}.

{Notice first that a sharp Serrin type result {\cite[p. 416]{Serrin}} for arbitrary constant $H$ was not established in every Cartan-Hadamard manifold (compare Theorems \ref{cond_nece_suf_minimo_Hadamard} and \ref{cond_nece_suf_minimo_curv_positivo_este_este}). However, we get a sharp Serrin criterion when $M$ is the hyperbolic space.} 

{Observe that} if $M=\HH^n$, it follows from the Spruck's existence result \cite[Th. 1.4 p. 787]{spruck} that the Serrin condition is a sufficient condition if $H\geq \frac{n-1}{n}$. In the opposite case $0<H<\frac{n-1}{n}$, Spruck noted that it was possible to establish an existence result if the strict inequality $(n-1)\Hc_{\partial\W} > n {H}$ holds. He used the entire graphs of constant mean curvature $\frac{n-1}{n}$ in $\HH^n\times\R$ as barriers (see \cite{BerardRicardo} for explicit formulas). However, this restriction over the Serrin condition in the last case does not allow to establish a Serrin type solvability criterion for every constant $H$ directly from Spruck's existence result {\cite[Th. 5.4 p. 797]{spruck}} when the ambient is the hyperbolic space.

We have established an existence result \cite[Th. 5 p. 4]{artigoexist:inpress}
%\cite[Th. 4.4 p. 51]{minhatese} 
for prescribed 
%$H\in\cl^{1,\alpha}(\overline{\W}\times\R)$ 
{mean curvature which extends the Spruck's existence result mentioned above for the hyperbolic space, and that also gives the following Serrin type solvability criterion when combined with Collorary \ref{M_nao_exist_MxR_Hxz_Hadamard}:}
\begin{teo}[Serrin type solvability criterion 1]\label{cond_nece_suf_HnxR}
Let $\W\subset \HH^n$ be a bounded domain with $\partial\W$ of class $\cl^{2,\alpha}$ for some $\alpha\in(0,1)$. Let $H\in\cl^{1,\alpha}(\overline{\W}\times\R)$ be a function satisfying $\Dz H\geq 0$ and $0 \leq {H} \leq \frac{n-1}{n}$ in ${\W\times\R}$. Then the Dirichlet problem for equation \eqref{operador_minimo_1} has a unique solution $u\in\cl^{2,\alpha}(\overline{\W})$ for every $\varphi\in\cl^{2,\alpha}(\overline\W)$ if, and only if, the strong Serrin condition \eqref{SerrinConditionGeneral} holds.
\end{teo}

By combining the existence result of Spruck {\cite[Th. 1.4 p. 787]{spruck}} with Corollary \ref{M_nao_exist_MxR_Hxz_Hadamard}, and putting together Theorem \ref{cond_nece_suf_HnxR}, we deduce that the sharp solvability criterion of Serrin {\cite[p. 416]{Serrin}} for arbitrary constant $H$ also holds in the $\cl^{2,\alpha}$ class if we replace $\R^n$ {by} $\HH^n$:
%%%%%%%%%%%%%%%%%%%
\begin{teo}[Sharp Serrin type solvability criterion]\label{cond_nece_suf_HnxR_constate}
Let $\W\subset \HH^n$ be a bounded domain whose boundary is of class {$\cl^{2,\alpha}$}. Then for every constant $H$ the Dirichlet problem for equation $\M u=nH$ %\eqref{operador_minimo_1} 
has a unique solution for arbitrary continuous boundary data if, and only if, $(n-1) \Hc_{\partial\W}(y) \geq n \modulo{H}$.
\end{teo}
%%%%%%%%%%%%%%%%%%%%%

%%%%%%%%%%%%%%%%
{We have also proved the following generalization of the Spruck's existence result {\cite[Th. 1.4 p. 787]{spruck}} for constant mean curvature:
%\cite[Th. 4.1 p. 40]{artigoexist:inpress, minhatese} 
\begin{teo}[{\cite[Th. 4 p. 4]{artigoexist:inpress}}]\label{T_exist_Ricci}
Let $\Omega \subset M$ be a bounded domain with $\partial\W$ of class $\cl^{2,\alpha}$ for some $\alpha\in(0,1)$. 
Let $H\in\cl^{1,\alpha}(\overline{\W}\times\R)$ satisfying $\Dz H \geq 0$ and
\begin{equation}\label{cond_H_Ricci_0}
\Ricc_x\geq n\sup\limits_{z\in\R}\norm{\nabla_x H(x,z)}-\dfrac{n^2}{n-1}\ds\inf_{z\in\R}\left(H(x,z)\right)^2 \ \forall \ x\in\W.
\end{equation}
If 
$$%\begin{equation}\label{StrongSerrinCondition}
(n-1)\Hc_{\partial\W}(y)\geq n \sup\limits_{z\in\R}\modulo{H\left(y,z\right)} \ \forall \ y\in\partial\W,
$$%\end{equation}
then for every $\varphi\in\cl^{2,\alpha}(\overline{\W})$ there exists a unique solution $u\in\cl^{2,\alpha}(\overline{\W})$ of the Dirichlet problem for equation \eqref{operador_minimo_1}.
\end{teo}

Theorem \ref{T_exist_Ricci} in combination with Corollaries \ref{M_nao_exist_MxR_Hxz_Hadamard} and \ref{M_nao_exist_MxR_Hxz_positive_curvature} yields the following generalization in the $\cl^{2,\alpha}$ class of a theorem of Serrin {\cite[Th. p. 484]{Serrin}} in the Euclidean space:
%%%%%%%%%%%%%%%%%%%%%%%%%%%%
\begin{teo}[Serrin type solvability criterion 2]\label{cond_nece_suf_hadamard}
Let $\W\subset M$ be a bounded domain whose boundary is of class $\cl^{2,\alpha}$ for some $\alpha\in(0,1)$. Suppose that $H\in\cl^{1,\alpha}(\overline{\W}\times\R)$ is either non-negative or non-positive in $\overline{\W}\times\R$, $\Dz H\geq 0$ and
$$\Ricc_x\geq n\ds\sup_{z\in\R}\norm{\nabla_x H(x,z)}-\dfrac{n^2}{n-1}\ds\inf_{z\in\R}\left(H(x,z)\right)^2, \ \forall \ x\in\W.$$
Assume either that
\begin{enumerate}
\item $M$ is a Cartan-Hadamard manifold, or
\item $M$ is a compact manifold whose sectional curvature $K$ satisfies $0<\frac{1}{4} K_0 < K \leq K_0$ and $\diam(\W)<\frac{\pi}{2\sqrt{K_0}}$.
\end{enumerate}
Then the Dirichlet problem for equation \eqref{operador_minimo_1} has a unique solution $u\in\cl^{2,\alpha}(\overline{\W})$ for every $\varphi\in\cl^{2,\alpha}(\overline{\W})$ if, and only if, the strong Serrin condition \eqref{SerrinConditionGeneral} holds.
\end{teo}
%%%%%%%%%%%%%%%%%%%%%%%%%%%%
%Finally, using Corollary \ref{M_nao_exist_MxR_Hxz_positive_curvature} we also obtain:
%\begin{teo}[Serrin type solvability criterion 3]
%\label{cond_nece_suf_curv_posit}
%Let $M$ be a complete and compact manifold whose sectional curvature satisfies $\frac{1}{4} K_0 < K \leq K_0$ for a positive constant $K_0$. Let $\W\subset M$ be a domain with $\diam(\W)<\frac{\pi}{2\sqrt{K_0}}$ and whose boundary is of class $\cl^{2,\alpha}$ for some $\alpha\in(0,1)$. Suppose that $H\in\cl^{1,\alpha}(\overline{\W}\times\R)$ is either non-negative or non-positive in $\overline{\W}\times\R$, $\Dz H\geq 0$ and
%$$\Ricc_x\geq n\ds\sup_{z\in\R}\norm{\nabla_x H(x,z)}-\dfrac{n^2}{n-1}\ds\inf_{z\in\R}\left(H(x,z)\right)^2, \ \forall \ x\in\W.$$
%Then the Dirichlet problem for equation \eqref{operador_minimo_1} has a unique solution $u\in\cl^{2,\alpha}(\overline{\W})$ for every $\varphi\in\cl^{2,\alpha}(\overline{\W})$ if, and only if, the strong Serrin condition \eqref{SerrinConditionGeneral} holds.
%\end{teo}
}%

\section{Proof of the main non-existence theorem}\label{cap_NaoExis}
\markboth{Y. N. Alvarez and R. Sa Earp}{Proof of the main non-existence theorem}
%\vspace*{-0.5cm}

%Lembramos que $M$ indicará uma variedade Riemanniana completa de dimensão $n$. Seja $\W\subset M$ domínio limitado e $H$ uma função contínua em $\overline{\W}\times\R$. O objetivo deste capítulo é mostrar que se a condição de Serrin forte
%\begin{equation}\label{cond_Serrin_forte_quitar}
%(n-1)\Hc_{\partial\W}(y)\geq n\ds\sup_{z\in\R}\modulo{H(y,z)} \ \forall\ y\in\partial\W
%\end{equation}
%não é satisfeita, então existem dados no bordo para os quais não existe nenhuma função $u\in\cl^2(\W)\cap\cl^0(\overline{\W})$ satisfazendo
%\begin{equation}\label{Prob_Dirichlet_NaoExist}
%\left\{ \begin{array}{ll}
%\M u=nH(x,u) &\mbox{em } \W,\\
%\phantom{\M}   u=\varphi &\mbox{em } \partial\W.
%\end{array} \right.
%\end{equation}
%

%\section{Domínios limitados}
The proof of Theorem \ref{M_nao_exist_MxR_Hxz} is based on two results that will be proved in the sequel. The following fundamental proposition traces its roots back to the work of Finn \cite[Lemma p. 139]{Finn1965} when he established the theorem ensuring the non-existence of solutions for Dirichlet problems for the minimal surface equation in non-convex domain of $\R^2$. His lemma was extended by Jenkins-Serrin {\cite[Prop. III p. 182]{Serrin1968}} for the minimal hypersurface equation in $\R^n$, and subsequently by Serrin \cite[Th. 1 p. 459]{Serrin} for quasilinear elliptic operators (see also \cite[Th. 14.10 p. 347]{GT}). Afterward M. Telichevesky \cite[Lemma. 11 p. 250]{miriam} extended the result for the minimal vertical equation in $M\times\R$. We will use some of the ideas of these works.
%Vamos formalizar este resultado para o operador $\Q$ sobre uma variedade Riemanniana qualquer seguindo as ideias dos trabalhos precedentes.

\begin{prop}\label{M_prop_gen_JS}
Let $\W\in M$ be a bounded domain. Let $\Gamma'$ be a relative open portion of $\partial\W$ of class $\cl^1$. Let $H\in\cl^{0}(\overline{\W}\times\R)$ be a function non-decreasing in the variable $z$. Let $u\in\cl^2(\W)\cap\cl^1(\W\cup \Gamma')\cap\cl^0(\overline{\W})$ and $v\in\cl^2(\W)\cap\cl^0(\overline{\W})$ satisfying
$$%\begin{equation}\label{cond_prop}
\left\{ \begin{array}{cl}
\Q u \geq \Q v &\mbox{in } \W,\\
   u \leq v &\mbox{in } \partial\W\setminus\Gamma',\\
	\parcial{v}{N}=-\infty &\mbox{in }\Gamma',
\end{array} \right.
$$%\end{equation}
where $N$ is the inner normal to $\Gamma'$. Under these conditions $u\leq v$ in $\Gamma'$. Therefore $u\leq v$ in $\W$.
\end{prop}
\begin{proof}
%Se $u\leq v$ em $\Gamma'$, então $u\leq v$ em $\W$ pelo princípio do máximo \ref{PM_quasilineares}.
By contradiction, suppose that $m=\ds\max_{\Gamma'} (u-v) > 0.$
Hence, $u \leq v + m $ in $\Gamma'$. Then $u\leq v+m$ in $\partial\W$ since $u\leq v$ in $\partial\W\setminus\Gamma'$ by hypotheses.
In view of the function $H$ is non-decreasing in $z$ and $m>0$, we have
$$\Q(v+m)=\M (v+m)-nH(x,v+m)\leq \M v - nH(x,v)=\Q v \leq \Q u.$$
As a consequence of the maximum principle (see \cite[Th. 10.1 p. 263]{GT}) $u\leq v+m$ in $\W$.
Let $y_0\in \Gamma'$ be such that $m = u(y_0)-v(y_0)$. Let $\gamma_{y_0}=\exp_{y_0}(tN_{y_0})$, for $t>0$ near $0$. Then
\begin{align*}
u(\gamma_{y_0}(t))-u(y_0)
%&=u(\gamma_{y_0}(t))-(v(y_0)+m)
\leq	 \left(v\left(\gamma_{y_0}(t)\right)+m\right)-\left(v(y_0)+m\right)= v(\gamma_{y_0}(t))-v(y_0).
\end{align*}
Dividing the expression by $t$ and passing to the limit as $t$ goes to zero it follows that $ \parcial{u}{N}\leq -\infty$. This is a contradiction since $u\in\cl^1(\Gamma')$, hence, $u\leq v$ in $\Gamma'$. %Em decorrência do princípio do máximo $u\leq v$ em $\W$.
\end{proof}

The next lemma plays a fundamental role in this paper. In this lemma it is established a height a priori estimate for solutions of equation $\M u=nH(x,u)$ in $\W$ in those points of $\partial\W$ on which the strong Serrin condition \eqref{SerrinConditionGeneral} fails.
\begin{lema}\label{M_nao_exist_MxR_estimativa_Hxz}
Let $\W\subset M$ be a bounded domain whose boundary is of class $\cl^2$. Let $H\in\cl^0(\overline{\W}\times\R)$ be a non-negative function and non-decreasing in the variable $z$, and $u\in\cl^2(\W)\cap\cl^0(\overline{\W})$ satisfying $\M u = nH(x,u)$.
Let us assume that there exists $y_0\in\partial\W$ such that
\begin{equation}\label{cond_Serrin_negac_M_pos}
(n-1)\Hc_{\partial\W}(y_0)<nH(y_0,k)
\end{equation}
for some $k\in\R$. Suppose also that $\cut(y_0)\cap\W=\emptyset$. Furthermore, assume that the radial curvature over the radial geodesics issuing from $y_0$ and intersecting $\W$ is bounded above by $K_0$, where
\begin{enumerate}[label=(\alph*)]
\item $K_0\leq 0$, or
\item $K_0>0$ and $\dist(y_0,x)<\dfrac{\pi}{2\sqrt{K_0}}$ for all $x\in\overline{\W}$.
\end{enumerate}
Then for each $\varepsilon>0$ there exists {a ball $B_a(y_0)$ centered at $y_0$ e radius $a>0$} depending only on $\varepsilon$, $\Hc_{\partial\W}(y_0)$, the geometry of $\W$ and the modulus of continuity of $H(x,k)$ in $y_0$,  such that
\begin{equation}\label{est_nao_exist_Hxz}
u(y_0) < \ds\max\left\{k,\ds\sup_{\partial\W\setminus B_a(y_0)}~u\right\}+\varepsilon. %\ \ \mbox{ou} \ \ u(y_0) > \ds\min\left\{-M,\ds\inf_{\partial\W\setminus B_a(y_0)} u\right\} - \varepsilon.
\end{equation}
\end{lema}

\begin{proof}
{The proof is done in two steps.} First, we will find an estimate for $u(y_0)$
depending on $k$ and $\sup\limits_{\partial B_a(y_0)\cap \W}u$ for some $a$ that does not depend on $u$. Secondly, we will get an upper bound for $\ds \sup_{\partial B_a(y_0)\cap \W}u$ in terms of $\sup\limits_{\partial\W\setminus B_a(y_0)} u $. %Finally, we will can reduce $a$ in terms of $\varepsilon$.

\bigskip
\noindent\textbf{Step 1.}

First of all note that, from \eqref{cond_Serrin_negac_M_pos}, there exists $\nu>0$ such that
\begin{equation}\label{M_condcomigual}
(n-1)\Hc_{\partial\W}(y_0) < n H(y_0,k)-4\nu.
\end{equation}
%%%%%%%%%%%%%%%%%%%%%
Let $R_1>0$ be such that $\partial B_{R_1}(y_0)\cap\W$ is connected and
\begin{equation}\label{M_cond_H}
\modulo{H(x,k)-H(y_0,k)}<\dfrac{\nu}{n}, \ \forall \ x\in B_{R_1}(y_0)\cap\W.
\end{equation}

%Isso segue do fato de $H(x,k)$ ser contínua em $\overline{\W}$.
%%%%%%%%%%%%%%%%%%%%%%%%

Note also that we can construct an embedded and oriented hypersurface $S$, tangent to $\partial\W$ at $y_0$ and whose mean curvature with respect to the normal $N$ pointing inwards $\W$ at $y_0$ satisfies
\begin{equation}\label{M_curv_S_Gamma_ponto}
%\Hc_{\partial\W}(y_0)<\Hc_{S}(y_0)<\Hc_{\partial\W}(y_0)+\dfrac{\nu}{(n-1)}.
\Hc_{S}(y_0)<\Hc_{\partial\W}(y_0)+\dfrac{\nu}{(n-1)}.
\end{equation}
%We note that we have choose the orientation of $S$ given by the normal which coincides with the inner normal to $\partial\W$ in $y_0$.

It is well known that for some $\tau>0$ the map
$$\funciones{\Phi_t}{S}{\W}{y}{\exp^{\bot}(y,tN_y)}$$
is a diffeomorphism for each $0\leq t < \tau$, and so $S_t:=\Phi_t(S)$ is parallel to $S$.
%(see {\cite[Obs. 1.2 p. 16]{marcos}, \cite[p. 119]{petersen1998}}).
{Moreover, the distance function $d(x)=\dist(x,S)$ is of class $\cl^2$ over the set
%$$\Sigma_{\tau}=\{x\in \W ;\ x= \Phi_t(y), \ y\in S , \ 0 \leq t <\tau\}.$$ %(veja lema \ref{lema_tau}).
$$\Sigma_{\tau}=\{\Phi_t(y), \ y\in S , \ 0 \leq t <\tau\}\subset\W.$$ %(veja lema \ref{lema_tau}).
}%
%%%%%%%%%%%%%%%%%%%%%%%%%%%%%%%
Let $0<R_2<\min\{\tau,R_1\}$ be such that
\begin{equation}\label{est_Laplaciano_d}
\modulo{\Delta d(x)-\Delta d(y_0)}<\nu \ \ \forall \ x\in B_{R_2}(y_0)\cap\Sigma_\tau.
%\modulo{\Hc_{S_t}(x)-\Hc_{S}(y_0)}<\frac{\nu}{n-1}, \ \forall \ t\in(0,R_2), \ \forall \ x \in S_t \cap B_{R_2}(y_0).
\end{equation}%

We now fix $a<R_2$. For $0<\epsilon<a$ we set
$$\W_{\epsilon}=\{x \in B_a(y_0)\cap\Sigma_\tau; d(x)>\epsilon\}.$$

{Let $\phi\in\cl^2(\epsilon, a)$ be a non-negative convex function, decreasing in $(0,a)$ and whose graph gets very steep as $t$ approaches $\epsilon$ from the right. That is, $\phi$ satisfies} 
{\small%
\begin{multicols}{4}
\begin{enumerate}
\item[P1.] $\phi(a)=0$,
\item[P2.] $\phi'\leq0$,
\item[P3.] $\phi''\geq0$,
\item[P4.] $\phi'(\epsilon)=-\infty$.
\end{enumerate}
\end{multicols}}%
\noindent We also require that $\phi'^3\nu + \phi''$=0 in $(\epsilon,a)$. 

\bigskip

Let $v = \ds\max\left\{k,\ds\sup_{\partial B_a(y_0)\cap\W} u \right\}  + \phi\circ d$. So, $v\geq u$ in $\partial \W_{\epsilon}\setminus S_{\epsilon}$.
In addition, if $N_{\epsilon}$ is the normal to $S_{\epsilon}$ inwards $\W_{\epsilon}$ and $x\in S_{\epsilon}\cap B_a(y_0)$, then
\begin{align*}
\parcial{v}{N_{\epsilon}}(x)&=\escalar{\nabla v(x)}{N_{\epsilon}(x)}=\escalar{\phi'(d(x))\nabla d(x)}{\nabla d(x)}=\phi'(\epsilon)=-\infty.
\end{align*}

On the other hand, for $x\in \W_{\epsilon}$, a straightforward computation yields
%$$\Q v =  \phi'(1+\phi'^2) \Delta d + \phi''-n H(x,v) (1+\phi'^2)^{3/2}.$$
$$\Q v =  \frac{\phi'}{(1+\phi'^2)^{1/2}} \Delta d + \dfrac{\phi''}{(1+\phi'^2)^{3/2}}-n H(x,v).$$
Since $v \geq k$ and $H$ is non-decreasing in $z$ it follows that $H(x,v)\geq H(x,k)$. Hence,
%$$\Q v \leq  \phi'(1+\phi'^2) \Delta d + \phi''-n H(x,k) (1+\phi'^2)^{3/2}.$$
$$\Q v \leq  \frac{\phi'}{(1+\phi'^2)^{1/2}} \Delta d + \dfrac{\phi''}{(1+\phi'^2)^{3/2}}-n H(x,k).$$
By means of the properties of $\phi$ we have
%$$(1+\phi'^2)^{3/2}
%%=(1+\phi'^2)^{1/2}(1+\phi'^2)
%>(\phi'^2)^{1/2}(1+\phi'^2)=\modulo{\phi'}(1+\phi'^2)=-\phi'(1+\phi'^2)$$
%$$(1+\phi'^2)^{1/2}>(\phi'^2)^{1/2}=\modulo{\phi'}=-\phi',$$
%temos
$$\dfrac{\phi'}{(1+\phi'^2)^{1/2}} >-1,$$
and by the assumption on the sign of $H$ we obtain
$$-nH(x,k) < nH(x,k)\dfrac{\phi'}{(1+\phi'^2)^{1/2}} .$$
Therefore,
\begin{equation}\label{M_exp_Qv_3}
\Q v < \frac{\phi'}{(1+\phi'^2)^{1/2}} \left(\Delta d(x) + n H(x,k)  \right) + \frac{\phi''}{(1+\phi'^2)^{3/2}}.
\end{equation}
Furthermore,
%\begin{align*}
%&\Delta d(x) + n H(x,k)  \\
%= & -(n-1)\Hc_{S_t}(x) + (n-1)\Hc_{S}(y_0) -(n-1)\Hc_{S}(y_0) + n H(x,k)\\
												%> & -\nu -(n-1)\Hc_S(y_0) + n H(x,k)\tag{a}\\
										    %> & -2\nu -(n-1)\Hc_{\partial\W}(y_0) + n H(x,k)\tag{b}\\
										    %> &  2\nu - n H(y_0,k) + n H(x,k) \tag{c}\\
										    %> & \nu \tag{d},
%\end{align*}
\begin{align*}
\Delta d(x) + n H(x,k) = & \Delta d(x) - \Delta d(y_0) + \Delta d(y_0) + n H(x,k)\\
												> & -\nu -(n-1)\Hc_S(y_0) + n H(x,k)\tag{a}\\
										    > & -2\nu -(n-1)\Hc_{\partial\W}(y_0) + n H(x,k)\tag{b}\\
										    > &  2\nu - n H(y_0,k) + n H(x,k) \tag{c}\\
										    > & \nu \tag{d},
\end{align*}
where (a) follows directly from \eqref{est_Laplaciano_d}, (b) from \eqref{M_curv_S_Gamma_ponto}, (c) from \eqref{M_condcomigual} and (d) from \eqref{M_cond_H}. Using this estimate on \eqref{M_exp_Qv_3} we have
\begin{align*}
\Q v <& \dfrac{\phi'}{(1+\phi'^2)^{1/2}}\nu+ \dfrac{\phi''}{(1+\phi'^2)^{3/2}}\\
 =&\dfrac{1}{(1+\phi'^2)^{3/2}}\left( \phi'(1+\phi'^2) \nu + \phi''\right) \\
<& \dfrac{1}{(1+\phi'^2)^{3/2}}\left( \phi'^3 \nu + \phi''\right).
\end{align*}
Then, $\Q v <0$ in $\W_{\epsilon}$ in view of the requirements on $\phi$. 

Choosing $\phi$ explicitly by\footnote{See also \cite[\S 14.4]{GT} and \cite[Th. 4.1 p. 40]{elias}.}
\begin{equation}\label{M_exp_phi}
\phi(t)=\sqrt{\dfrac{2}{\nu}}\left((a-\epsilon)^{1/2}-(t-\epsilon)^{1/2}\right),
\end{equation}
we observe that $\phi$ satisfies P1--P4 and that $\phi'^3 \nu + \phi''=0$ in $(\epsilon,a)$.
%{%\color{red}
%In fact,
%$$ \phi'(t)=-\dfrac{1}{2}\sqrt{\dfrac{2}{\nu}}(t-\epsilon)^{-1/2}=-\dfrac{1}{2}\left(\dfrac{2}{\nu(t-\epsilon)}\right)^{1/2}$$
%and
%$$ \phi''(t)=\dfrac{1}{4}\sqrt{\dfrac{2}{\nu}}(t-\epsilon)^{-3/2}=\dfrac{\nu}{8}\left(\dfrac{2}{\nu(t-\epsilon)}\right)^{3/2}=-\nu \phi'(t)^3.$$
%}%
From Proposition \ref{M_prop_gen_JS} we deduce then
$$ u \leq v= \ds\max\left\{k,\ds\sup_{\partial B_a(y_0)\cap\W} u \right\} + \phi(\epsilon) \ \ \mbox{in}\ \ S_{\epsilon}\cap B_a(y_0). $$
In particular,
$$ u(\exp_{y_0}(\epsilon N_{y_0})) \leq \ds\max\left\{k,\ds\sup_{\partial B_a(y_0)\cap\W} u \right\} + \sqrt{\dfrac{2}{\nu}}\left((a-\epsilon)^{1/2}\right).$$
%where $\gamma_{y_0}(\epsilon)=\exp_{y_0}(\epsilon N_{y_0})$. 
Since this estimate holds for each $0<\epsilon<a$, we can pass to the limit as $\epsilon$ goes to zero to obtain
\begin{equation}\label{est_u0_1}
u(y_0) \leq  \ds\max\left\{k,\ds\sup_{\partial B_a(y_0)\cap\W} u \right\} + \sqrt{\dfrac{2a}{\nu}}.
\end{equation}

%%%%%%%%%%%%%%%%%%%%%%

\bigskip

\noindent\textbf{Step 2.}

{Let $\delta=\diam(\W)$. Analogously to step 1, we require a function $\psi\in\cl^2(a,\delta)$, non-negative and convex, decreasing in $(a,\delta)$ and whose graph is very steep near $a$. That is, 
{\small%
\begin{multicols}{4}
\begin{enumerate}
\item[P5.] $\psi(\delta)=0$,
\item[P6.] $\psi'\leq0$,
\item[P7.] $\psi''\geq0$,
\item[P8.] $\psi'(a)=-\infty$,
\end{enumerate}
\end{multicols}}%
\noindent In addition, we need that $\frac{c\psi'^3}{t}+\psi''\leq 0$ in $(a,\delta)$ for a positive constant $c$ to be chosen later on.}

Let $w=\ds\sup_{\partial\W\setminus B_a(y_0)} u + \psi\circ\rho$ be defined in $\W'=\W\setminus B_a(y_0)$, where $\rho(x)=\dist(x,y_0)$. We remind that $\rho\in\cl^2(M\setminus(\cut(y_0)\cup\{y_0\}))$, so $w\in\cl^2(\W\setminus B_a(y_0))$. % (veja proposição \ref{dist_2_regularidade}).
The idea is to use Proposition \ref{M_prop_gen_JS} again.
We note that $w\geq u$ in $\partial\W\setminus B_a(y_0)$.
Also, if $N_a$ is the normal field to $\partial B_a(y_0)\cap\W$ inwards $\W'$, we have for each $x\in\partial B_a(y_0)\cap\W$ that
\begin{align*}
\parcial{w}{N_a}(x)&=\escalar{\nabla w(x)}{N_a(x)}=\escalar{\psi'(\rho(x))\nabla \rho(x)}{\nabla \rho(x)}=\psi'(a)=-\infty.
\end{align*}
For $w$ we have
$$ \Q w =  \dfrac{\psi'}{(1+\psi'^2)^{1/2}} \Delta \rho + \dfrac{\psi''}{(1+\psi'^2)^{3/2}}-n H(x,w) .$$
Since $H\geq 0$, it follows
$$ \Q w \leq  \dfrac{\psi'}{(1+\psi'^2)^{1/2}} \Delta \rho + \dfrac{\psi''}{(1+\psi'^2)^{3/2}}.$$
In any of the hypothesis (a) or (b), the radial geodesics issuing from $y_0$ and intercepting $\W$ do not contain conjugate points to $y_0$ (see {\cite[Th. 6.5.6 p. 151]{Klingenberg}, \cite[Th. p. 107]{Chavel}}). Then the Laplacian comparison theorem {\cite[Th. A p. 19]{greenewu}} can be used to estimate $\Delta \rho$ {in $\W'$}.

Under the hypothesis (a) we compare $M$ with $\R^n$ to obtain
$$\Delta \rho(x) \geq \dfrac{n-1}{\rho(x)}.$$
%já que conhecemos a expressão do Laplaciano em $\R^n$ (veja proposição \ref{hessianorho}).

Under the hypothesis (b) we compare $M$ with the sphere $S^n_{K_0}$ of sectional curvature $K_0>0$. In this case
$$%\begin{equation}\label{Laplaciano_rho_comp_positivo}
\Delta \rho(x) \geq (n-1)\sqrt{K_0}\cot\left(\sqrt{K_0}\rho(x)\right).
$$%\end{equation}
From the second assumption on (b) there also exists $0<\kappa<\frac{\pi}{2\sqrt{K_0}}$ such that $\dist(x,y_0)\leq \frac{\pi}{2\sqrt{K_0}}-\kappa$, for each $x\in\overline{\W}$. Thus, for each $x\in\W\setminus B_a(y_0)$, there exists a unique normal minimizing geodesic $\beta$ such that $\beta(0)=y_0$ and $\beta (t_0)=x$, where $t_0\leq\dfrac{\pi}{2\sqrt{K_0}}-\kappa$. Let us define the function $\xi(t)=\sqrt{K_0}t\cot\left(\sqrt{K_0}t\right)$ for $t>0$. We note that $\xi$ is decreasing and $\xi\left(\frac{\pi}{2\sqrt{K_0}}\right)=0$. Then,
$$\xi(t)\geq \xi\left(\dfrac{\pi}{2\sqrt{K_0}}-\kappa\right)>0, \ \forall t\in\left.\left( 0, \dfrac{\pi}{2\sqrt{K_0}}-\kappa\right.\right].$$
%Portanto,
%$$\sqrt{K_0}\cot\left(\sqrt{K_0}t\right) \geq \dfrac{\sqrt{K_0}\left(\dfrac{\pi}{2\sqrt{K_0}}-\kappa\right)\cot\left(\sqrt{K_0}\left(\dfrac{\pi}{2\sqrt{K_0}}-\kappa\right)\right)}{t}=\dfrac{C}{t}.$$
Consequently,
\[ \rho(x)\Delta \rho(x) \geq (n-1)C,\]
%onde
%%Desta forma,
%\begin{equation}\label{Laplaciano_rho_comp_positivo_kappa}
%\Delta \rho(x) \geq \dfrac{(n-1)C}{\rho(x)},
%\end{equation}
where $$C=\sqrt{K_0}\left(\frac{\pi}{2\sqrt{K_0}}-\kappa\right)\cot\left(\sqrt{K_0}\left(\frac{\pi}{2\sqrt{K_0}}-\kappa\right)\right)>0.$$

Thus $\Delta \rho(x)\geq \frac{c}{\rho}$, where $c=n-1$ in the case (a) and $c=(n-1)C$ in the case (b).
Therefore, 
%$$\Q w \leq  \dfrac{c}{\rho} \psi'(1+\psi'^2)  + \psi'' < \dfrac{c}{\rho} \psi'^3  + \psi''.$$
\begin{align*}
 \Q w \leq & \dfrac{\psi'}{(1+\psi'^2)^{1/2}} \cdot \dfrac{c}{\rho} + \dfrac{\psi''}{(1+\psi'^2)^{3/2}}\\
=&\dfrac{1}{(1+\psi'^2)^{3/2}}\left(\dfrac{c}{\rho} \psi'(1+\psi'^2)  + \psi''\right)\\
<&\dfrac{1}{(1+\psi'^2)^{3/2}}\left(\dfrac{c}{\rho} \psi'^3  + \psi''\right).
\end{align*}
So, $\Q w <0$ in $\W'$ due to the construction of $\psi$.

Let us define $\psi$ as \footnote{See also \cite[\S 14.4]{GT}}
\begin{equation}\label{M_expres_psi}
%\psi(t)=\left(\dfrac{2}{n-1}\right)^{1/2}\ds\int_t^{\delta} \left(\log \frac{r}{a}\right)^{-1/2}dr,
\psi(t)=\left(\dfrac{2}{c}\right)^{1/2}\ds\int_t^{\delta} \left(\log \frac{r}{a}\right)^{-1/2}dr.
\end{equation}
Such a function satisfies P5--P8, and also $\dfrac{c}{t} \psi'(t)^3  + \psi''(t)<0$ in $(a,\delta)$.
%{\color{red}
%De fato,
%$$\psi'(t)=-\left(\dfrac{2}{c}\right)^{1/2}\left(\log \frac{t}{a}\right)^{-1/2},$$
%e
%\begin{align*}
%\psi''(t)=&-\left(\dfrac{2}{c}\right)^{1/2}\left(-\dfrac{1}{2}\left(\log \frac{t}{a}\right)^{-3/2}\dfrac{a}{t}\dfrac{1}{a}\right)\\
         %=&\dfrac{1}{2t}\left(\dfrac{2}{c}\right)^{1/2}\left(\log \frac{t}{a}\right)^{-3/2}\\
				 %=&\dfrac{c}{4t}\left(\dfrac{2}{c}\right)^{3/2}\left(\log \frac{t}{a}\right)^{-3/2}\\
				 %=&-\dfrac{c}{4t}\psi'(t)^3\\
				 %<&-\dfrac{c}{t}\psi'(t)^3.
%\end{align*}
%}
From Proposition \ref{M_prop_gen_JS} we can conclude that $u \leq w$ in $\partial B_a(y_0)\cap\W$, and then
%\begin{equation}\label{M_est_etapa1}
%u \leq w = \ds\sup_{\partial\W\setminus B_a(y_0)} u^+ + \psi(a)\ \ \mbox{em}\ \ \partial B_a(y_0)\cap\W
%\end{equation}
%e, portanto,
\begin{equation}\label{M_desigualdade_bordo}
%\ds\sup_{\partial B_a(y_0)\cap \W} u \leq \ds\sup_{\partial\W\setminus B_a(y_0)} u + \psi(a).
\ds\sup_{\partial B_a(y_0)\cap \W} u \leq \ds\sup_{\partial\W\setminus B_a(y_0)} u + \psi(a).
\end{equation}

{We remark that in step 2 no geometric property on $a$ is required other than the connectedness of $\partial B_a(y_0)\cap \W$.}

Finally, we use \eqref{M_desigualdade_bordo} in \eqref{est_u0_1} from step 1, so
$$u(y_0) \leq \ds\max\left\{k,\ds\sup_{\partial\W\setminus B_a(y_0)} u\right\} + \psi(a) +\sqrt{\dfrac{2a}{\nu}}.$$
It is easy to see that $\ds\lim_{a\rightarrow 0}\psi(a)=0$.
%Com efeito, da expressão \eqref{M_expres_psi} vê-se que
%$$\psi(a)=\left(\dfrac{2}{c}\right)^{1/2}\ds\int_a^{\delta} \left(\log \frac{r}{a}\right)^{-1/2}dr,$$
%Da mudança de variável $r=a e^s$ decorre
%\begin{equation*}%\label{expres_psi_a}
%\psi(a)=a\left(\dfrac{2}{c}\right)^{1/2}\ds\int_0^{\log(\delta/a)} \dfrac{e^s}{\sqrt{s}}ds.
%\end{equation*}
%Integrando por partes verifica-se
%$$\ds\int_0^{\log\frac{\delta}{a}} \dfrac{e^s}{\sqrt{s}}ds= 2\sqrt{s}{e^s}\Big|_0^{\log(\delta/a)}-2\ds\int_0^{\log\frac{\delta}{a}} \sqrt{s}e^sds.$$
%e da mudança de variável $r=a e^{s^2}$ decorre
%\begin{equation*}%\label{expres_psi_a}
%\psi(a)=2a\left(\dfrac{2}{c}\right)^{1/2}\ds\int_0^{\sqrt{\log(\frac{\delta}{a})}} {e^{s^2}}ds.
%\end{equation*}
%Integrando por partes verifica-se
%$$\ds\int_0^{\log\frac{\delta}{a}} \dfrac{e^s}{\sqrt{s}}ds= 2\sqrt{s}{e^s}\Big|_0^{\log(\delta/a)}-2\ds\int_0^{\log\frac{\delta}{a}} \sqrt{s}e^sds.$$
Hence, for each $\varepsilon>0$, $a$ can be chosen small enough to satisfy
\[%\begin{equation}\label{depenciaepsilon_Hneg}
\psi(a)+\sqrt{\dfrac{2a}{\nu}}<\varepsilon.\qedhere
\]%\end{equation}
%\begin{equation}\label{M_est_u_y0_ugeqM}
%u(y_0) < \ds\max\left\{k,\ds\sup_{\partial\W\setminus B_a(y_0)} u\right\}  + \varepsilon,
%\end{equation}
\end{proof}
{%%%
\begin{obs}
The constant $\frac{\pi}{2\sqrt{K_0}}$ in item (b) of the statement of Theorem \ref{M_nao_exist_MxR_Hxz} is essential for the technique we have used in the proof of Lemma \ref{M_nao_exist_MxR_estimativa_Hxz}. However, it seems that this constant can be improved to $\frac{\pi}{\sqrt{K_0}}$. 
\end{obs}%
\begin{obs}
In the case where $H$ is a function that does not depends on the height variable, then the estimate \eqref{est_nao_exist_Hxz} becomes
$$%\begin{equation}\label{est_nao_existH(x)posit}
u(y_0) < \ds\sup_{\partial\W\setminus B_a(y_0)} u + \varepsilon.
$$%\end{equation}
%where $a$ is constructed as before. 
\end{obs}}%

At last we are able to prove Theorem \ref{M_nao_exist_MxR_Hxz}.

%\begin{proof}
\bigskip
\noindent\textit{Proof of the main non-existence theorem}.
Obviously we can suppose that $H\geq 0$. Then,
$$(n-1)\Hc_{\partial\W}(y_0)<nH(y_0,k)$$
for some $k\in\R$ since $H$ is non-decreasing in $z$. Let $\varepsilon>0$ and $\varphi\in\cl^{\infty}(\overline{\W})$ such that $\varphi=k$ in $\partial\W\setminus B_a(y_0)$ and $\varphi(y_0)=k+\varepsilon$. Hence, no solution of equation \eqref{operador_minimo_1} in $\W$ could have $\varphi$ as boundary values because such a function does not satisfy the estimate \eqref{est_nao_exist_Hxz}. \hfill \qedsymbol

\bibliographystyle{plain}
%\nocite{*}
\nocite{artigoexist:inpress}
\bibliography{bibliografia}

\vfill

\noindent Yunelsy N Alvarez\\
Pontifícia Universidade Católica do Rio de Janeiro\\
Departamento de Matemática\\
Rio de Janeiro\\
CEP 22451-900\\
Brazil

\smallskip

\noindent Current Institution:\\
Universidade de São Paulo\\
Instituto de Matemática e Estadística\\
Departamento de Matemática\\
São Paulo\\
CEP 05508-090\\
Brazil

\smallskip

\noindent Email address: ynapolez@gmail.com; ynalvarez@usp.br

\bigskip\medskip

\noindent Ricardo Sa Earp\\
Pontifícia Universidade Católica do Rio de Janeiro\\
Departamento de Matemática\\
Rio de Janeiro\\
CEP 22451-900\\
Brazil\\
Email address: rsaearp@gmail.com\\
\end{document}